\documentclass[11pt]{article}
\usepackage{amsmath}
\usepackage{amsthm}
\usepackage{latexsym}
\usepackage{graphicx, epsfig}
\usepackage{amssymb,amsmath}  
\usepackage{multimedia}
\usepackage{subfigure}
\usepackage{mathrsfs}
\usepackage{mathdots}
\usepackage{cases}

\usepackage{pgfpages,tikz,tikz-cd,pgfkeys,pgfplots}
\usetikzlibrary{arrows,positioning,matrix,fit,backgrounds,shapes}

\usetikzlibrary{fit,matrix}
\tikzset{
mN/.style = {
    draw=#1, semithick, inner sep=0pt}
             }

\definecolor{LemonChiffon}{rgb}{100, 98, 80}
\definecolor{myblue}{rgb}{0,0.4,0.8}
\definecolor{orange}{rgb}{1, 0.4, 0}
\definecolor{mygreen}{rgb}{0, 0.8, 0}
\definecolor{myred}{rgb}{204, 0, 0}
\definecolor{violet}{RGB}{0.4,0.2,1}
\definecolor{brown}{rgb}{0.6, 0.4, 0}

\makeatletter
\newtheorem*{rep@XThm}{\rep@title}
\newcommand{\newreptheorem}[2]{%
\newenvironment{rep#1}[1]{%
 \def\rep@title{#2 \ref{##1}}%
 \begin{rep@XThm}}%
 {\end{rep@XThm}}}
\makeatother

\newtheorem{theorem}{Theorem}[section]
\newtheorem{XThm}[theorem]{Theorem}
\newreptheorem{XThm}{Theorem}
\newtheorem{lemma}[theorem]{Lemma}
\newtheorem{proposition}[theorem]{Proposition}
\newtheorem{corollary}[theorem]{Corollary}

\theoremstyle{definition}
\newtheorem{definition}[theorem]{Definition}
\newtheorem{remark}[theorem]{Remark}

\newtheorem{example}[theorem]{Example}

\title{Matrix periods and competition periods of Boolean Toeplitz matrices}
\date{}
\author{Gi-Sang Cheon$^{a, b}$, Bumtle Kang$^{b}$, Suh-Ryung Kim$^{b, c}$, and Homoon Ryu$^{b,c}$ \\
{\footnotesize $^a$ \textit{Department of Mathematics, Sungkyunkwan
University, Suwon 16419, Rep. of Korea}}\\
{\footnotesize $^{b}$ \textit{Applied Algebra and Optimization
Research Center, Sungkyunkwan University,}}\\{\footnotesize\textit{
Suwon 16419, Rep. of Korea}}\\
 {\footnotesize $^{c}$ \textit{Department of Mathematics Education,
Seoul National University,}}\\{\footnotesize\textit{
Seoul 08826, Rep. of Korea}}\\
{\footnotesize gscheon@skku.edu, lokbt@hotmail.com, srkim@snu.ac.kr, and ryuhomun@naver.com}}

\begin{document}

\maketitle

\begin{abstract}
In this paper, we study the matrix period and the competition period of Toeplitz matrices over a binary Boolean ring $\mathbb{B} = \{0,1\}$. Given subsets $S$ and $T$ of $\{1,\ldots,n-1\}$, an $n\times n$ Toeplitz matrix $A=T_n\langle S ; T \rangle$ is defined to have $1$ as the $(i,j)$-entry if and only if $j-i \in S$ or $i-j \in T$. We show that if $\max S+\min T \le n$ and $\min S+\max T \le n$, then $A$ has the matrix period $d/d'$ and the competition period $1$ where $d = \gcd (s+t \mid s \in S, t \in T)$ and $d' = \gcd(d, \min S)$. Moreover, it is shown that the limit of the matrix sequence $\{A^m(A^T)^m\}_{m=1}^\infty$ is a directed sum of matrices of all ones except zero diagonal.
In many literatures we see that graph theoretic method can be used to prove strong structural properties about matrices. Likewise, we develop our work from a graph theoretic point of view.
\end{abstract}

\section{Introduction}


A {\em binary Boolean ring} $(\mathbb{B}, +, \cdot)$ is a set $\mathbb{B} = \{0,1\}$ with two binary operations $+$ and $\cdot$ on $\mathbb{B}$ defined by
\[
\begin{array}{c|cc}
+ & 0 & 1 \\
\hline
0 & 0 & 1 \\
1 & 1 & 1
\end{array}  \quad \text{ and }\quad
\begin{array}{c|cc}
\cdot & 0 & 1 \\
\hline
0 & 0 & 0 \\
1 & 0 & 1
\end{array}
\]

In this paper, we consider {\em Boolean matrices} with entries from a binary Boolean ring, and the set of $n \times n$ Boolean matrices is denoted by $\mathbb{B}_n$.
Take $A \in \mathbb{B}_n$.
The {\em matrix period} of $A$ is the
smallest positive integer $p$ for which there is a positive integer $M$ such that $A^m = A^{m+p}$ for any integer $m \ge M$.
We note that the rows $i$ and $j$ of $A^m$ have a common nonzero entry in some column if and only if the $(i,j)$-entry of $A^m(A^T)^m$ is $1$.
Consider the matrix sequence $\{A^m(A^T)^m\}_{m=1}^\infty$.
Since $|\mathbb{B}_n|= 2^{n^2}$, there is the smallest positive integer $q$ such that \[A^{q+i}(A^T)^{q+i}=A^{q+r+i}(A^T)^{q+r+i}\] for some positive integer $r$ and every nonnegative integer $i$.
Then there is also the smallest positive integer $p$ such that $A^{q}(A^T)^q=A^{q+p}(A^T)^{q+p}$.
Those integers $q$ and $p$ are called the {\it competition index} and {\it competition period} of $A$, respectively, which was introduced by Cho and Kim~\cite{cho2013competition}.
Refer to \cite{cho2011competition,kim2008competition,kim2010generalized, kim2015characterization, kim2012bound} for further results on competition indices and competition periods of digraphs.

 An $n\times n$ matrix $A\in\mathbb{B}_n$ is called {\it primitive} if $A^m$ is the all ones matrix for some integer $m\ge1$. The minimum such $m$ is known as the
{\it exponent} of $A$ denoted by ${\rm exp}(A)$. Clearly, if $A\in\mathbb{B}_n$ is primitive, then its matrix period and competition period are 1. Wielandt's theorem \cite{Bru} states that ${\rm exp}(A)\le (n-1)^2+1$ for an $n\times n$ primitive matrix $A\in\mathbb{B}_n$.
Moreover, there is a primitive matrix $W\in\mathbb{B}_n$ such that ${\rm exp}(W)=(n - 1)^2+1$. Recently, it is shown \cite{exp} that the matrix $W$ is permutation equivalent to the (0,1)-Toeplitz matrix for odd $n\ge6$.
In this paper, we extend the work done in \cite{exp} by studying the matrix period and the competition period of Toeplitz matrices in $\mathbb{B}_n$.

A (0,1)-matrix $A=(a_{ij})\in \mathbb{B}_n$ is called a {\em Boolean Toeplitz matrix} if $a_{ij}=a_{j-i} \in \mathbb{B}$, {\it i.e.} $A$ is of the Toeplitz form:
\[
\begin{bmatrix}
a_0 & a_1 & \cdots & a_{n-1} \\
a_{-1} & a_0  & \ddots &\vdots \\
\vdots & \ddots & \ddots  & a_1  \\
a_{-n+1}&\cdots &a_{-1} & a_0
\end{bmatrix}.
\]
Accordingly, a Boolean Toeplitz matrix $A\in \mathbb{B}_n$ is determined by two nonempty subsets $S$ and $T$, not necessarily disjoint, of $\{1,\ldots, n-1\}$ so that $a_{ij}=1$ if and only if $j-i \in S$ or $i-j \in T$.
We assume that $S=\{s_1,\ldots,s_{k_1}\}$ and $T=\{t_1,\ldots,t_{k_2}\}$ where
$$
1\le s_1<\ldots<s_{k_1}<n\quad{\rm and}\quad 1\le t_1<\ldots<t_{k_2}<n.
$$
Note that $S=\{j\mid a_j=1\}$ and $T=\{i\mid a_{-i}=1\}$.
In this context, we denote a Boolean Toeplitz matrix $A$ associated with index sets $S$ and $T$ by $T_n\langle s_1,\ldots,s_{k_1};t_1,\ldots,t_{k_2}\rangle$ or simply by $T_n\langle S;T\rangle$.

Consider the Boolean Toeplitz matrix
\[
A =T_5\langle 2 ; 4 \rangle= \begin{bmatrix}
0 & 0 & 1 & 0 & 0 \\
0 & 0 & 0 & 1 & 0\\
0 & 0 & 0 & 0 & 1 \\
0 & 0 & 0 & 0 & 0 \\
1 & 0 & 0 & 0 & 0
\end{bmatrix}.
\]
Then
\[
A^{3k-1} = \begin{bmatrix}
0 & 0 & 0 & 0 & 1 \\
0 & 0 & 0 & 0 & 0 \\
1 & 0 & 0 & 0 & 0 \\
0 & 0 & 0 & 0 & 0 \\
0 & 0 & 1 & 0 & 0
\end{bmatrix},
\quad
A^{3k} = \begin{bmatrix}
1 & 0 & 0 & 0 & 0 \\
0 & 0 & 0 & 0 & 0 \\
0 & 0 & 1 & 0 & 0 \\
0 & 0 & 0 & 0 & 0 \\
0 & 0 & 0 & 0 & 1
\end{bmatrix},
\quad
A^{3k+1} = \begin{bmatrix}
0 & 0 & 1 & 0 & 0 \\
0 & 0 & 0 & 0 & 0 \\
0 & 0 & 0 & 0 & 1 \\
0 & 0 & 0 & 0 & 0 \\
1 & 0 & 0 & 0 & 0
\end{bmatrix}
\]
for any integer $k \ge 1$.
Thus $A^m$ is not a Toeplitz matrix for any integer $m \ge 2$.
This happening cannot occur for a Toeplitz matrix $T_n \langle S; T \rangle$ with the property that $\max S + \min T \le n$ and $\min S + \max T \le n$ (Theorem~\ref{thm:power}).

In many literatures we see that graph theoretic method can be used to prove strong structural properties about matrices. We proceed this work from a graph theoretic point of view. The {\em support} of an $n\times n$ matrix $A=(a_{ij})$ is the $(0,1)$-matrix with $a_{ij}=1$ whenever $a_{ij}\ne0$.
If the support of $A$ is the adjacency matrix of a digraph $D$  then $D$ is called the {\em digraph of $A$}.
A sequence of $m$ arcs of the form $(v_0,v_1),(v_1,v_2),\ldots,(v_{m-1},v_m)$ is called a {\it directed $(v_0,v_m)$-walk of length $m$. This walk is also denoted by
$$
v_0\rightarrow v_1\rightarrow v_2\rightarrow \cdots \rightarrow v_m.
$$}
 Now let $D$ be a digraph with $n$ vertices. For a positive integer $m$,  a graph is called the {\it $m$-step competition graph} \cite{mstepcomp} of $D$, provided that it has the same vertex set as $D$ and $\{u,v\}$ is an edge if and only if there is a vertex $w$ in $D$ such that there are directed $(u,w)$-walk of length $m$ and directed $(v,w)$-walk of length $m$. We denote the $m$-step competition graph of $D$ by $C^m(D)$.
In particular, the $1$-step competition graph of $D$ is called the {\it competition graph} of $D$ and it is denoted by $C(D)$.


Throughout this paper, we focus on Boolean Toeplitz matrices $A=T_n \langle S;T\rangle$ satisfying the conditions $\max S+\min T \le n$ and $\min S+\max T \le n$. See Remark~\ref{rem:most} to see why those are the most that we may consider.
The digraph $D$ of an $n\times n$ Toeplitz matrix $T_n \langle S;T\rangle$ is assumed to have the vertex set $[n]$, the $n$-set $\{1,\ldots,n\}$. By definition, it is obvious that if $i<j$ then $(i,j)$ is an arc in $D$ if and only if $j-i\in S$, and if $i>j$ then $(i,j)$ is an arc in $D$ if and only if $i-j\in T.$ For example, see Figure 1.

As results, in Section 2 we show that they have the matrix period $d/d'$ and the competition period is $1$ where $d = \gcd (s+t \mid s \in S, t \in T)$ and $d' = \gcd(d, \min S)$ (Theorem~\ref{thm:exp}).
It is worth noting that the competition period of a Toeplitz matrix $A$ is $1$ no matter whether $A$ is primitive.
It is also proved that the limit of the matrix sequence $\{A^m(A^T)^m\}_{m=1}^\infty$ is a directed sum of matrices of all ones except zero diagonal for $A$ with $\max S+\min T \le n$ and $\min S+\max T \le n$ (Corollary~\ref{cor:directsum}).
Theorem~\ref{lem:pqr} plays a key role in proving these two main results and will be separately proven in Section~\ref{se:pf}.
Section~\ref{sec:walks} is devoted to building a directed $(u,v)$-walk having a designated number of certain types of arcs, which will be used as a useful tool in Section~\ref{se:pf}.
Finally, we give an upper bound for the competition index of a specific type of Boolean Toeplitz matrices (Theorem~\ref{thm:compindex}).

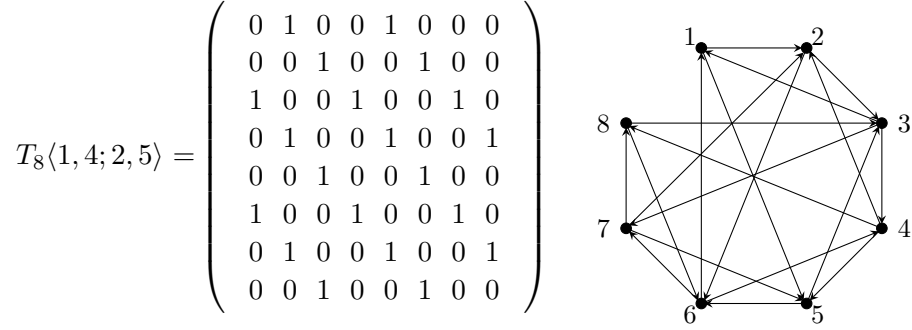
\begin{figure}
\begin{tikzpicture}[>=stealth,thick,baseline]
\node (matrix) [matrix of math nodes,left delimiter=(,right delimiter={)}]
{0 & 1 & 0 & 0 & 1 & 0 & 0 & 0 \\
0 & 0 & 1 & 0 & 0 & 1 & 0 & 0 \\
1 & 0 & 0 & 1 & 0 & 0 & 1 & 0 \\
0 & 1 & 0 & 0 & 1 & 0 & 0 & 1 \\
0 & 0 & 1 & 0 & 0 & 1 & 0 & 0 \\
1 & 0 & 0 & 1 & 0 & 0 & 1 & 0 \\
0 & 1 & 0 & 0 & 1 & 0 & 0 & 1 \\
0 & 0 & 1 & 0 & 0 & 1 & 0 & 0 \\};
\node [left=3mm of matrix]{$T_8 \langle 1,4;2,5\rangle=$};
\end{tikzpicture}
\quad
\begin{tikzpicture}[baseline={(0,3.2)}]
\filldraw (-0.7,4.65) circle (2pt);
\draw (-0.85,4.8) node{\small $1$};
\filldraw (0.7,4.65) circle (2pt);
\draw (0.85,4.8) node{\small $2$};
\filldraw (1.7,3.65) circle (2pt);
\draw (2.0,3.65) node{\small $3$};
\filldraw (1.7,2.25) circle (2pt);
\draw (2.0,2.25) node{\small $4$};
\filldraw (0.7,1.25) circle (2pt);
\draw (0.85,1.1) node{\small $5$};
\filldraw (-0.7,1.25) circle (2pt);
\draw (-0.85,1.1) node{\small $6$};
\filldraw (-1.7,2.25) circle (2pt);
\draw (-2.0,2.25) node{\small $7$};
\filldraw (-1.7,3.65) circle (2pt);
\draw (-2.0,3.65) node{\small $8$};

\draw[->,>=stealth] (-0.7,4.65) -- (0.64,4.65);
\draw[->,>=stealth] (-0.7,4.65) -- (0.66,1.29);

\draw[->,>=stealth] (0.7,4.65) -- (1.66,3.69);
\draw[->,>=stealth] (0.7,4.65) -- (-0.66,1.29);

\draw[->,>=stealth] (1.7,3.65) -- (1.7,2.31);
\draw[->,>=stealth] (1.7,3.65) -- (-1.66,2.29);
\draw[->,>=stealth] (1.7,3.65) -- (-0.66,4.61);

\draw[->,>=stealth] (1.7,2.25) -- (0.74,1.29);
\draw[->,>=stealth] (1.7,2.25) -- (-1.66,3.61);
\draw[->,>=stealth] (1.7,2.25) -- (0.74,4.61);

\draw[->,>=stealth] (0.7,1.25) -- (-0.66,1.25);
\draw[->,>=stealth] (0.7,1.25) -- (1.66,3.61);

\draw[->,>=stealth] (-0.7,1.25) -- (-1.66,2.21);
\draw[->,>=stealth] (-0.7,1.25) -- (1.66,2.21);
\draw[->,>=stealth] (-0.7,1.25) -- (-0.7,4.59);

\draw[->,>=stealth] (-1.7,2.25) -- (-1.7,3.59);
\draw[->,>=stealth] (-1.7,2.25) -- (0.66,1.29);
\draw[->,>=stealth] (-1.7,2.25) -- (0.66,4.59);

\draw[->,>=stealth] (-1.7,3.65) -- (-0.74,1.29);
\draw[->,>=stealth] (-1.7,3.65) -- (1.64,3.65);

\end{tikzpicture}

\caption{A Toeplitz matrix $T_8\langle 1,4;2,5\rangle$ and its digraph.}
\label{fig:t81245}
\end{figure}

\section{Main Results}\label{sec:pre}

Recall that $T_n\langle S;T\rangle$ represents an $n\times n$ Boolean Toeplitz matrix with
$$
S=\{s_1,\ldots,s_{k_1}\}\quad \mbox{ and } \quad T=\{t_1,\ldots,t_{k_2}\}
$$
where $s_1=\min S$, $s_{k_1}=\max S$, $t_1=\min T$, and $t_{k_2}=\max T$.
From now on, we assume that $|S| = k_1$ and $|T| = k_2$ for $T_n \langle S;T \rangle$ unless otherwise mentioned.
\vskip.5pc

The following lemma gives a necessary condition for the existence of a directed walk from a certain vertex to a certain vertex in the digraph of a Toeplitz matrix (it is restated under our definition of the directed graph of a Toeplitz matrix, which is isomorphic to that of the digraph of a Toeplitz matrix in \cite{exp}).

\begin{lemma}[\cite{exp}]\label{lem:cwcorr} Let $D$ be the digraph of $T_n\langle S;T\rangle$ and $W$ be a $(u,v)$-directed walk of length $m$ in $D$. Then there are nonnegative integer sequences
$(a_{i})_{i=1}^{k_1}$ and $(b_i)_{i=1}^{k_2}$ such that
\begin{eqnarray*}
 \sum_{i=1}^{k_1} a_i s_i - \sum_{i=1}^{k_2} b_i t_i = v-u \quad
\mbox{and} \quad m = \sum_{i=1}^{k_1}a_i + \sum_{i=1}^{k_2}
b_i.
\end{eqnarray*}\end{lemma}

We will examine $A^m(A^T)^m=[b_{ij}]$ for $A = T_n\langle S;T\rangle$ in a graph theoretic view.
 Recall that $b_{ij}=1$ if and only if the rows $i$ and $j$ of $A^m$ have a common nonzero entry in some column. Thus $b_{ij}=1$ if and only if for some vertex $v$ in $D$ there exist a directed $(u_i,v)$-walk and a directed $(u_j,v)$-walk of length $m$  where $u_i$ and $u_j$ are vertices of $D$ corresponding to the rows $i$ and $j$ of $A$, that is, $u_i$ and $u_j$ are adjacent in $C^m(D)$. Hence, by Lemma~\ref{lem:cwcorr}, there are some nonnegative integer sequences $(a_i)_{i=1}^{k_1}$, $(b_i)_{i=1}^{k_2}$, $(c_i)_{i=1}^{k_1}$, and $(d_i)_{i=1}^{k_2}$ such that
\begin{eqnarray}\label{v-u}
v-u_i = \sum_{i=1}^{k_1} a_i s_i - \sum_{i=1}^{k_2} b_it_i, \quad v-u_j = \sum_{i=1}^{k_1} c_i s_i - \sum_{i=1}^{k_2} d_it_i,
\end{eqnarray}
and
\begin{eqnarray}\label{ell}
m = \sum_{i=1}^{k_1}a_i+\sum_{i=1}^{k_2}b_i = \sum_{i=1}^{k_1}c_i+\sum_{i=1}^{k_2}d_i.
\end{eqnarray}
Since $a_i s_i$ is written as a repeated sum $s_i+\cdots+s_i$ of $a_i$ summands (similarly, $b_it_i$, $c_i s_i$, $d_it_i$),
it follows from (\ref{v-u}) and (\ref{ell}) that
the number of total summands of $s_i$'s and $-t_i$'s representing $v-u_1$ and $v-u_2$ are the same as $m$.
Therefore
\begin{itemize}
\item[($\star$)] $u_2-u_1=(v-u_1)-(v-u_2)$ can be represented by a linear combination (with integer coeffcients) of elements in the set
\begin{eqnarray*}\label{u}
 U_{S,T}:=\{s_j - s_i>0 \mid s_i,s_j \in S\} \cup \{t_j-t_i>0 \mid t_i,t_j \in T\} \cup \{s+t \mid s \in S, t\in T\}.
\end{eqnarray*}
\end{itemize}

Let ${\rm gcd}(U_{S,T})$ and ${\rm gcd}(S+T)$ denote the greatest common divisors of the elements in the set $U_{S,T}$ and $\{s+t \mid s \in S, t\in T\}$, respectively.
From now on, we assume that the sets $S$ and $T$ for $\gcd(S+T)$ are from the Toeplitz matrix $A_n \langle S;T\rangle$.
Then we have the following proposition.
\begin{proposition}\label{prop:gcd}
For $S, T \subseteq [n-1]$, ${\rm gcd}(U_{S,T})={\rm gcd}(S+T)$.
\end{proposition}
\begin{proof} For brevity, let $d={\rm gcd}(U_{S,T})$ and $d'={\rm gcd}(S+T)$. By definition, $d'\mid d$. We show that $d\mid d'$. Let $i_1, i_2 \in [k_1]$ be given. Since $d \mid (s_{i_1}+t_1)$ and $d \mid (s_{i_2}+t_1)$, we have $d \mid (s_{i_1}-s_{i_2})$.
Similarly, $d \mid (t_{j_1}-t_{j_2})$ for each $j_1, j_2 \in [k_2]$. Thus $d \mid d'$.
\end{proof}

By $(\star)$ and Propostition~\ref{prop:gcd}, we have the following theorem.

\begin{theorem}\label{prop:adj}
Let $D$ be the digraph of $T_n\langle S;T \rangle$.
If two vertices $u$ and $v$ are adjacent in $C^m(D)$ for some positive integer $m$, then $u-v$ is a multiple of ${\rm gcd}(S+T)$.
\end{theorem}
The converse of Theorem~\ref{prop:adj} is also true for sufficiently large $m$ as long as $\max S + \min T \le n$ and $\min S + \max S \le n$ by the following theorem, which is one of our main theorems.

\begin{XThm}\label{thm:exp}
Let $D$ be the digraph of $T_n\langle S;T \rangle$ with $\max S + \min T \le n$, $\min S + \max S \le n$, $d = {\rm gcd}(S+T)$, and $d' = \gcd (d,s_1)$.
Then
\begin{itemize}
\item[(a)] the matrix period of $D$ is $d/d'$;
\item[(b)] the competition period of $D$ is $1$;
\item[(c)] the limit of $\{C^m(D)\}_{m=1}^\infty$ is the disjoint union of cliques $\{v \in [n] \mid v \equiv i \pmod d\}$ over all $i$ such that $1 \le i \le d$.
\end{itemize}
\end{XThm}

Theorem~\ref{thm:exp} can be restated in the viewpoint of matrix as follows.

\begin{corollary}\label{cor:directsum}
Let $A=T_n\langle S;T\rangle$ be a Boolean Toeplitz matrix with $\max S + \min T \le n$, $\min S + \max S \le n$, $d = \gcd(S+T)$, and $d' = \gcd (d,s_1)$.
Then
\begin{itemize}
\item[(a)] the matrix period of $A$ is $d/d'$;
\item[(b)] the competition period of $A$ is $1$;
\item[(c)] there exists a permutation $P$ such that \[
P\left(A^m(A^T)^m  \right)P^T \to
\left[\begin{array}{ccc}
J_{m_1} &  & O \\
 & \ddots &  \\
O &  & J_{m_d}
\end{array}\right]\text{ as } m\rightarrow \infty
\]where $J_{m_i}$ is the all-ones matrix of order $m_i:=|\{v \in [n] \mid v \equiv i \pmod d\}|$.
\end{itemize}
\end{corollary}

To prove the above theorem, we need to introduce the following sets.
For a positive integer $n$, we denote the interval $[-n+1, n-1]$ by $\mathcal{I}_n$.
\begin{definition}
For nonempty sets $S, T \subseteq [n-1]$ and a positive integer $i$, we introduce the following sets:
\begin{itemize}
\item $P_i = \{\ell \in \mathcal{I}_n \mid \ell \equiv is_1 \pmod d\}$ where $d = \gcd(S+T)$;
\item $Q_i = \{\sum_{j=1}^{k_1} a_js_j - \sum_{j=1}^{k_2} b_j t_j \in \mathcal{I}_n\mid a_j,b_j \in \mathbb{Z}^+_0, \sum_{j=1}^{k_1} a_j + \sum_{j=1}^{k_2} b_j = i\}$;
\item $R_i$ be the set of $\ell \in\mathcal{I}_n$ such that, for any vertices $u$ and $v$ with $v-u = \ell$, there exists a directed $(u,v)$-walk of length $i$ in the digraph of $T_n \langle S;T \rangle$.
\end{itemize}
\end{definition}

 It is immediately true by Lemma~\ref{lem:cwcorr} that $R_i \subseteq Q_i$ for any positive integer $i$.
Moreover, it is a direct consequence of the following proposition that $Q_i \subseteq P_i$ for any positive integer $i$.
Thus
\begin{equation}\label{eq:pqr}
R_i \subseteq Q_i \subseteq P_i
\end{equation}
for any positive integer $i$.

\begin{proposition}\label{lem:xiperiod}
	For nonempty sets $S, T \subseteq [n-1]$, let $d = {\rm gcd}(S+T)$.
	Then for any integers $a_i$, $b_j$,
	\[\sum_{i=1}^{k_1}a_is_i - \sum_{i=1}^{k_2}b_it_i \equiv \left(\sum_{i=1}^{k_1}a_i+\sum_{i=1}^{k_2}b_i\right) s_1 \pmod d.\]
\end{proposition}

\begin{proof}
	By the definition of $d$,
	\begin{align*}
		\sum_{i=1}^{k_1}a_is_i - \sum_{i=1}^{k_2}b_it_i&=\sum_{i=1}^{k_1}a_i(s_i+t_1-t_1) - \sum_{i=1}^{k_2} b_i(t_i+s_1-s_1) \\
		&\equiv \sum_{i=1}^{k_1}a_i(-t_1) -\sum_{i=1}^{k_2}b_i(-s_1) \pmod d \\
		&\equiv \sum_{i=1}^{k_1}a_i(-t_1+(s_1+t_1)) - \sum_{i=1}^{k_2} b_i(-s_1) \pmod d \\
		&= \left( \sum_{i=1}^{k_1}a_i +\sum_{i=1}^{k_2}b_i \right) s_1. \qedhere
	\end{align*}
\end{proof}

\begin{example}
Let $D$ be the digraph of $T_8 \langle 1,4 ; 2, 5 \rangle$.
Then $\gcd(S+T) = 3$ and $P_3 = \{-6, -3, 0, 3, 6\}$.
Since
\begin{gather*}
-6 = (0 \times s_1 + 0 \times s_2) - (3 \times t_1 + 0 \times t_2); \quad -3 = (1 \times s_1 + 0 \times s_2) - (2 \times t_1 + 0\times t_2); \\
0 = (2\times s_1 + 0\times s_2) - (1\times t_1 + 0\times t_2); \quad 3 = (3 \times s_1 + 0 \times s_2) - ( 0 \times t_1 + 0 \times t_2); \\
\quad 6 = (2 \times s_1 + 1 \times s_2) - (0 \times t_1 + 0 \times t_2),
\end{gather*}
we have $Q_3 \supseteq \{-6, -3, 0, 3, 6\}$.
Then $P_3 = Q_3$ by \eqref{eq:pqr}.

To check whether $\ell$ is an element of $R_3$ or not, one may check existence of a directed $(u,v)$-walk of length $3$ for every pair $(u,v)$ with $v-u = \ell$.
For $\ell = -3$, $(4,1), (5,2), (6,3), (7,4), (8,5)$ are the pairs $(u,v)$ with $v- u = \ell$ and
$\ell \in R_3$ since there are directed walks in $D$ such as
\begin{gather*}
4 \rightarrow 5 \rightarrow 3  \rightarrow 1; \quad 5 \rightarrow 6 \rightarrow 4 \rightarrow 2; \quad
6 \rightarrow 7 \rightarrow 5 \rightarrow 3;\\ 7 \rightarrow 8 \rightarrow 6 \rightarrow 4;
 \quad 8 \rightarrow 6 \rightarrow 4 \rightarrow 5.
\end{gather*}
Similarly, one can check to conclude that $R_3 \supseteq \{ -6, -3, 0, 3, 6\}$.
Therefore $P_3 = Q_3 = R_3$.
Indeed, there exists a sufficiently large positive integer $M_D$ such that if $i \ge M_D$, then $P_i = Q_i = R_i$, by Theorem~\ref{lem:pqr} below.
One may guess that $P_i = Q_i = R_i$ for every positive integer $i$.
Yet, not necessarily $P_i= Q_i = R_i$ for every positive integer $i$ as one can show that
$P_1 = \{ -5, -2, 1, 4, 7\} \supsetneq \{-5, -2, 1, 4\} = R_1$.
\end{example}

Theorem~\ref{thm:exp}, which is our main result, can be proven by the following two assertions.

 \begin{lemma}\label{lem:abcd}
For nonempty sets $S, T \subseteq[n-1]$, let $d=\gcd(S+T)$ and $d' = \gcd (s_1,d)$.
Then the following are true for any positive integer $i$:
\begin{itemize}
\item[(a)] $P_i = P_{i+d/d'}$;
\item[(b)] $P_i,\ldots, P_{i-1+d/d'}$ are mutually disjoint;
\item[(c)] $P_i = \{\ell \in \mathcal{I}_n \mid \ell-s_1 \in P_{i-1} \text{ or } \ell+t_1 \in P_{i-1}\} $ for any $i \ge 2$.
\end{itemize}
\end{lemma}
\begin{proof}
Fix a positive integer $i$. By the definition of $P_i$, $(i+d/d')s_1 = is_1 + (s_1/d')d \equiv is_1 \pmod d$, so $P_i = P_{i+d/d'}$. Therefore (a) holds.
To show (b), suppose $(i+j) s_1 \equiv (i+k) s_1 \pmod d$ for $\{j,k\} \subseteq [d/d']$.
Then $d \mid (k-j) s_1$, so $d/d' \mid (k-j)(s_1/d')$.
Since $d/d'$ and $s_1/d'$ are relatively prime, $d/d' \mid k-j$ and so $i+j=i+k$.
Therefore $P_1,\ldots,P_{d/d'}$ are mutually disjoint.
Finally, we show (c). If $\ell \in P_{i-1}$, then $\ell+s_1 \equiv (i-1)s_1 + s_1 \equiv is_1 \pmod d$ and $\ell-t_1 \equiv (i-1)s_1 -t_1 \equiv i s_1 \pmod d$.
Thus
 \[\big( \{\ell + s_1 \in \mathcal{I}_n  \mid \ell \in P_{i-1}\} \cup \{\ell -t_1 \in \mathcal{I}_n  \mid \ell\in P_{i-1} \} \big)\subseteq P_i.\]
Now take $m \in P_i$.
Then $-n+1 \le m \le n-1$,
\[m-s_1 \equiv is_1 - s_1 \equiv (i-1)s_1 \pmod d,\]
and
\[m+t_1 \equiv is_1 - s_1 \equiv (i-1)s_1 \pmod d.\]
Therefore, if $m \in [-n+1,0]$, then \[m+t_1 \ge -n+1 \quad \text{and} \quad m+t_1 \le t_1   \le n-1\]
and so $m+t_1 \in P_{i-1}$, i.e. $m \in
\{\ell - t_1\in \mathcal{I}_n  \mid \ell \in P_{i-1} \}$;
if $m \in P_i \cap [0,n-1]$, then
\[m-s_1 \ge -s_1 \ge -t_1+1 \ge -n+1 \quad \text{and} \quad m-s_1 \le n-1 \]
and so $m-s_1 \in P_{i-1}$, i.e. $\ell \in \{\ell+s_1\in \mathcal{I}_n   \mid \ell \in P_{i-1} \}$.
Thus \[P_i \subseteq  \{\ell + s_1\in \mathcal{I}_n  \mid \ell \in P_{i-1}\} \cup \{\ell -t_1 \in \mathcal{I}_n \mid \ell \in P_{i-1} \}.\]
Hence, $P_i = \{\ell \in \mathcal{I}_n  \mid \ell-s_1 \in P_{i-1} \text{ or } \ell+t_1\in P_{i-1}\}$.
\end{proof}
The following theorem plays a key role in proving Theorem~\ref{thm:exp}.
The next two sections will be devoted to proving it.

\begin{theorem}\label{lem:pqr}
 Let $D$ be the digraph of $T_n\langle S;T\rangle$ with $\max S + \min T \le n$, $\min S + \max S \le n$, and $d = {\rm gcd}(S+T)$.
Then there is a positive integer $M$ such that $P_i = Q_i=R_i$ for any integer $i \ge M$.
\end{theorem}

Theorem~\ref{lem:pqr} also enables us to prove Theorem~\ref{thm:power} mentioned earlier.

\begin{proof}[Proof of Theorem~\ref{thm:exp}]
Take two vertices $u,v$ in $D$.
Then $u-v \in \mathcal{I}_n $.
Since $s_{k_1} \le n-t_1$ and $t_{k_2} \le n-s_1$, it follows from Theorem~\ref{lem:pqr} that
there is a positive integer $M$ such that, for any integer $i \ge M$,
\begin{equation} \label{eq:PR}
P_i=R_i.
\end{equation}
Then $v-u \equiv i s_1 \pmod d$ if and only if there is a directed $(u,v)$-walk of length $i$ for any vertices $u$ and $v$ in $D$.

Take an integer $\ell \ge M$.
To show (a), we note that $P_\ell = P_{\ell + d/d'}$ by Lemma~\ref{lem:abcd}(a).
Then, by \eqref{eq:PR}, $R_\ell = R_{\ell + d/d'}$.
Thus $D^\ell = D^{\ell+d/d'}$.

Take a positive integer $p < d/d'$. Then $P_{\ell}$ and $P_{\ell+p}$ are disjoint by Lemma~\ref{lem:abcd}(b).
Thus there is an integer $x \in \mathcal{I}_n$ such that $a \in P_\ell \setminus P_{\ell+p}$, that is, $a \equiv\ell s_1 \pmod d$ and $a \not \equiv (\ell+p)s_1 \pmod d$.
Since $a \in \mathcal{I}_n$, there are some vertices $u, v$ in $D$ such that $v-u = a$.
Since $P_\ell = R_\ell$ and $P_{\ell+p}=R_{\ell+p}$, there is a directed $(u,v)$-walk of length $\ell$ while there is no directed $(u,v)$-walk of length $\ell+p$.
Thus $D^\ell \ne D^{\ell+p}$.
Hence the matrix period of $D$ is $d/d'$.

To show (b) and (c), we take two vertices $u$ and $v$ in $D$.
We may assume $u < v$.
Suppose $u \equiv v \pmod d$.
By the division algorithm, $\ell s_1 =dq+r$ for some integers $q$ and $r$ with $0 \le r \le d-1$.
Let $w=v-r$. Then $v-w = \ell s_1 - dq$ and $u < u+d \le v-r \le v$.
Since $u \equiv v \pmod d$, $v-w \equiv u-w \equiv \ell s_1 \pmod d$.
Thus, by \eqref{eq:PR}, there exist a directed $(v,w)$-walk of length $\ell$ and a directed $(u,w)$-walk of length $\ell$ in $D$.
Therefore $u$ and $v$ are adjacent in $C^\ell(D)$.
Thus \[\{v \in V(D) \mid v \equiv i \pmod d\}\] is a clique in $C^\ell(D)$ for $i=1, \ldots, d$.
If $u \not \equiv v \pmod d$, then $u$ and $v$ are not adjacent in $C^\ell(D)$ by Proposition~\ref{prop:adj}.
Hence $C^\ell(D)$ is a disjoint union of $d$ cliques $\{v \in V(D) \mid v \equiv 1 \pmod d\}, \ldots, \{v \in V(D) \mid v \equiv d \pmod d\}$.
Consequently, we have shown that the competition period of $D$ is $1$ and the limit of $\{C^m(D)\}_{m=1}^\infty$ is  the disjoint union of cliques $\{v \in V(D) \mid v \equiv 1 \pmod d\}, \ldots, \{v \in V(D) \mid v \equiv d \pmod d\}$.
\end{proof}

Given a digraph $D$ and a positive integer $m$, we denote by $D^m$ the digraph with vertex set same as $D$ and an arc $(u,v)$ if and only if there is a directed $(u,v)$-walk of length $m$.

\begin{theorem}\label{thm:power}
Let $A = T_n \langle S; T \rangle$ be a Boolean Toeplitz matrix with \[\max S + \min T \le n\quad \text{and} \quad \min S + \max T \le n.\]
Then there exists some positive integer $M$ such that, for every positive integer $m> M$, $A^m$ is a Toeplitz matrix.
\end{theorem}

\begin{proof}
By Theorem~\ref{lem:pqr}, there is a positive integer $M$ such that $P_i = R_i$ for any integer $m \ge M$.
Suppose that $(u,v)$-entry of $A^m$ is $1$ for some integer $m \ge M$.
Then there exists directed $(u,v)$-walk of length $m$ in the digraph $D$ of $A$.
Therefore $v-u \equiv ms_1 \pmod d$ by Lemma~\ref{lem:cwcorr} and Proposition~\ref{lem:xiperiod}.
Thus $v - u \in P_m$.
Since $P_m = R_m$, $v-u \in R_m$.
Then, there is a directed $(u',v')$-walk of length $m$ in $D$ for any $v'-u' = v-u$.
Therefore the diagonal containing $(u,v)$ consists of ones.
Since $(u,v)$ was arbitrarily chosen, $A^m$ is Toeplitz.
\end{proof}

\section{Building specific directed walks in the digraph of $T_n\langle S;T \rangle$}
\label{sec:walks}

For notational convenience, we call an arc $u \to v$ of $D$ an {\it $s_i$-arc} (resp.\ {\it $t_j$-arc})
 if $v - u = s_i$ (resp.\ $u-v = t_j$) for $i \in [k_1]$ (resp.\ $j \in [k_2]$).

This section is devoted to building a directed $(u,v)$-walk having
a designated number of $s_i$-arcs for each $i = 2, 3, \ldots, k_1$ and a designated number of $t_j$-arcs for each $j = 2, 3, \ldots, k_2$.
In addition, we give an upper bound for the competition index of a Toeplitz matrix $T_n\langle S;T \rangle$ with $s_{k_1} \le n-t_1$, $t_{k_2} \le n-s_1$.

To initiate Theorem~\ref{lem:walk}, we take the digraph $D$ of $T_8 \langle 1,4 ;2,5 \rangle$ given in Figure~\ref{fig:t81245} and consider the vertex $7$.
We wish to construct a directed walk from $7$ containing exactly five arcs $(u,v)$ with $v-u=4$ and exactly six arcs $(u,v)$ with $u-v=5$ (it may contain arcs $(u,v)$ with $v-u=1$ or $u-v=2$ as many as is desired).
We begin with a directed walk $W_0 = 7$.
Using arcs $(u,v)$ with $u-v=2$, we reach the vertex $1$, which is the smallest vertex that can be reached from $7$ by using arcs $(u,v)$ with $u-v=2$.
It can be done by the directed path $W_1 = 7 \rightarrow 5 \rightarrow 3 \rightarrow 1$.
Now we add an $(u,v)$ arc with $v-u=4$ to have $W_2: = W_1 \rightarrow 5$.
To realize another $(u,v)$ arc with $v-u = 4$, we take $W_3 := W_2 \rightarrow 3\rightarrow 1$ by using arcs $(u,v)$ with $u-v=2$ and add the arc $(1,5)$ at the end.
We repeat this process until we achieve five arcs $(u,v)$ with $v-u= 4$.
We observe that this procedure works because $s_2+t_1 = 4+2 \le 8 = n$.
By repeatedly reaching the vertex $8$, which is the largest vertex that can be reached from arcs $(u,v)$ with $v-u = 1$, by going through a similar process as above,
we may eventually obtain a desired directed walk.
This idea is generalized to Theorem~\ref{lem:walk}.

\begin{lemma}[\cite{compToep}]\label{lem:edgecond}
Let $D$ be the digraph of $T_n \langle S;T\rangle$ and take vertices $u$ and $b$ in $D$ with $u<v$.
Then $u$ and $v$ are adjacent in $C(D)$ if and only if one of the following is true:
\begin{numcases}
{v-u=}
s_i-s_j &  for some  $s_i \in S \cap [n-u]$ and  $s_j \in S \cap [n-v]$;  \label{eq:case1}\\
t_i-t_j & for some $t_i \in T_{k_2} \cap [v-1]$ and  $t_j \in T_{k_2} \cap  [u-1]$;\label{eq:case3} \\
s_i+t_j & for some $s_i \in S$ and $t_j \in T_{k_2}$.  \label{eq:case2}
\end{numcases}
\end{lemma}

\begin{lemma}[\cite{exp}]\label{lem:divis} Let $S$ and $T$ be nonempty subsets of $[n-1]$ with $n \ge \max S + \max T$.
For an integer $m\in [n]$ and any nonzero sequences
$(a_{i,m})_{i=1}^{k_1}$ and $(b_{i,m})_{i=1}^{k_2}$ of nonnegative integers  satisfying
\begin{equation} 1 \le m+ \sum_{i=1}^{k_1} a_{i,m}s_{i} -
\sum_{j=1}^{k_2} b_{j,m} t_{j} \le n, \label{eq:mrange}\end{equation}
there is an integer sequence  of length $\sum_{i=1}^{k_1}a_{i,m}+
\sum_{j=1}^{k_2}b_{j,m}$ satisfying the following:
\begin{itemize}
\item each term equals $s_i$ or $t_j$ for some $i \in [k_1]$ and $j \in [k_2]$ and $s_i$ and $t_j$ appear $a_{i,m}$ times and $b_{j,m}$ times, respectively, for each $i \in [k_1]$ and $j \in [k_2]$;
\item its $k$th partial sum is between $1-m$ and $n-m$ for any $1 \le k \le \sum_{i=1}^{k_1}a_{i,m}+
\sum_{j=1}^{k_2}b_{j,m}$.
\end{itemize}
\end{lemma}

The original statement of Lemma~\ref{lem:divis} includes the condition that $S$ and $T$ are disjoint.
However, the condition was not used in the proof of Lemma~\ref{lem:divis} and is deleted in the current statement.

\begin{theorem}\label{lem:walk}
Let $D$ be the digraph of $T_n\langle S;T\rangle$ with
$\max S + \min T \le n$ and $\min S + \max S \le n$. Suppose that a vertex $v$ of $D$ and nonnegative integers $a_2, \ldots, a_{k_1}, b_2, \ldots, b_{k_2}$ are given. Then the following are true:
\begin{itemize}
\item[(a)] There is a directed walk $W$ of length $\ell$ starting from $v$ and containing exactly $a_i$ $s_i$-arcs and $b_j$ $t_j$-arcs for each $2 \le i \le k_1$ and $2\le j\le k_2$ where
\[
\ell \le (a_2+\ldots +a_{k_1}+b_2+\ldots+b_{k_2})\left(\max\left\{\left\lceil\frac{t_{k_2}}{s_1}\right\rceil,\left\lceil\frac{s_{k_1}}{t_1}\right\rceil\right\}+1\right).
\]
\item[(b)]Given $W$ in part (a), if $a \ge a_1$ and $b \ge b_1$ satisfy the inequalities
\[
1 \le v+ as_1 + \sum_{i=2}^{k_1} a_is_i- bt_1- \sum_{j=2}^{k_2}b_jt_j \le n,
\]where $a_1$ and $b_1$ are the numbers of $s_1$-arcs and $t_1$-arcs, respectively, in $W$.
Then there is a directed walk  starting from $v$ in $D$ containing
exactly $a$ $s_1$-arcs, $b$ $t_1$-arcs, $a_i$ $s_i$-arcs and $b_j$ $t_j$-arcs for all $2\le i\le k_1$
and $2\le j \le k_2$.
\end{itemize}
\end{theorem}

\begin{proof}
We prove (a) by induction on $m:=a_2+\cdots+a_{k_1}+b_2+\cdots+b_{k_2}$
for some nonnegative integers $a_2,\ldots,a_{k_1},b_2,\ldots,b_{k_2}$.
If $m=0$, then the walk $v$ contains neither $s_i$-arcs nor $t_j$-arcs
for each $2 \le i \le k_1$ and $2 \le j \le k_2$.
Suppose $m\ge1$ and assume that the statement is true for $m-1$.

Now take nonnegative integers  $a_2,\ldots,a_{k_1},b_2,\ldots,b_{k_2}$ with
$a_2+\cdots+a_{k_1}+b_2+\cdots+b_{k_2}=m$.
Since $m \ge 1$, there is a positive integer $p$ or $q$ such that $a_p \ge 1$ or $b_q \ge 1$.
By symmetry, we may assume that $a_p \ge 1$.
Since $a_2+\cdots+a_{p-1}+(a_p-1)+a_{p+1}+\cdots+a_{k_1}+b_2+\cdots+b_{k_2}=m-1$,
by the induction hypothesis, there is a directed walk $W$ of length $\ell$ starting from $v$ such that \[\ell \le (m-1)\left(\max\left\{\left\lceil\frac{t_{k_2}}{s_1}\right\rceil,\left\lceil\frac{s_{k_1}}{t_1}\right\rceil\right\}+1\right)\]
 and $W$ contains exactly $a_p-1$ $s_p$-arcs, $a_i$ $s_i$-arcs, and $b_j$ $t_j$-arcs for
$2 \le i \le k_1$, $i \ne p$, $2 \le j \le k_2$.

Let $w$ be the terminus of $W$.
Then there exists a nonnegative integer $r$ such that $1 \le w -rt_1 \le t_1$.
Since $s_p \le s_{k_1}\le n-t_1$, we have $1\le w-rt_1 \le n-s_p$.
Now, let $r_0$ be the smallest one satisfying $1 \le w-r_0t_1 \le n-s_p$.
Then \[r_0 = \max\left\{0, \left\lceil \frac{w-n+s_p}{t_1}\right\rceil\right\} \le \left\lceil \frac{s_{k_1}}{t_1}\right\rceil.\]
Moreover, $W_1 := w \rightarrow (w-t_1) \rightarrow \cdots \rightarrow (w-r_0 t_1)$ is
a directed walk from $w$ to $w-r_0 t_1$ in $D$ consisting of only $r_0$ $t_1$-arcs.
Since $w-r_0 t_1\le n-s_p$, $w-r_0t_1+s_p \le n$ and so $w-r_0t_1+s_p$ is a vertex in $D$.
Then  $(w-r_0 t_1) \rightarrow (w-r_0 t_1 +s_p)$ in $D$, so the directed walk $W \rightarrow W_1 \rightarrow (w-r_0t_1+s_p)$ is a desired one starting from $v$
containing exactly $a_i$ $s_i$-arcs and $b_j$ $t_j$-arcs for $2 \le i \le k_1$ and
$2 \le j \le k_2$ and of length
\begin{align*}
\ell + r_0 +1 &\le (m-1)\left(\max\left\{\left\lceil\frac{t_{k_2}}{s_1}\right\rceil,\left\lceil\frac{s_{k_1}}{t_1}\right\rceil\right\}+1\right) + \left\lceil \frac{s_{k_1}}{t_1}\right\rceil +1 \\
&\le m \left(\max\left\{\left\lceil\frac{t_{k_2}}{s_1}\right\rceil,\left\lceil\frac{s_{k_1}}{t_1}\right\rceil\right\}+1\right). \end{align*}

To show (b), suppose $a \ge a_1$, $b \ge b_1$, and
\[ 1 \le
v+ as_1 + \sum_{i=2}^{k_1} a_is_i - bt_1 - \sum_{j=2}^{k_2} b_jt_j \le n.
\]
We note that the terminus of $W$ is
\[v+\sum_{i=1}^{k_1}a_is_i-\sum_{j=1}^{k_2}b_jt_j.\]
Then a desired directed walk can be obtained by attaching to $W$ a directed walk $W'$  containing $a-a_1$ $s_1$-arcs and $b-b_1$ $t_1$-arcs.
By applying Lemma~\ref{lem:divis} to $T_n \langle s_1;t_1 \rangle$ and $v+\sum_{i=1}^{k_1} a_is_i -  \sum_{j=1}^{k_2}b_jt_j$, we may obtain $W'$.
\end{proof}

\begin{remark} \label{rem:most}
If $s_{k_1} > n-t_1$ or $t_{k_2} > n-s_1$, there is no guarantee that there is a directed walk fulfilling the condition given in Theorem~\ref{lem:walk}(a).
For example, the digraph of $T_6 \langle 2, 4; 4,5 \rangle$ has two strong components with vertex sets $\{1,3,5\}$ and $\{2,4,6\}$ without arcs going from $\{1,3,5\}$ to $\{2,4,6\}$ and $(u,v)=(6,1)$ is the only arc with $u-v=5$ in $D$ (see Figure~\ref{fig:t62445}).
Therefore there is no directed walk containing arc $(u,v)$ with $u-v = 5$ starting from $1$.
\end{remark}

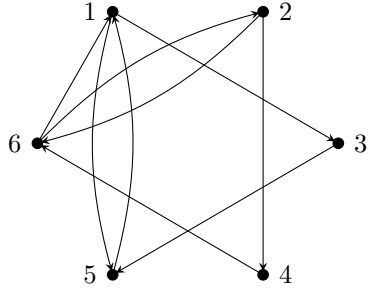
\begin{figure}\begin{center}
\begin{tikzpicture}[baseline={(0,3)}]
\coordinate (1) at (-1, 4.75);
\coordinate (2) at (1, 4.75);
\coordinate (3) at (2, 3);
\coordinate (4) at (1, 1.25);
\coordinate (5) at (-1, 1.25);
\coordinate (6) at (-2, 3);

\draw [ ->, shorten <= 0.05cm, shorten >= 0.05cm, >=stealth] (1) to (3);
\draw [ ->, shorten <= 0.05cm, shorten >= 0.05cm, >=stealth] (5) to[out=75,in=-75] (1);
\draw [ ->, shorten <= 0.05cm, shorten >= 0.05cm, >=stealth] (1) to[out=-105,in=105] (5);
\draw [ ->, shorten <= 0.05cm, shorten >= 0.05cm, >=stealth] (3) to (5);
\draw [ ->, shorten <= 0.05cm, shorten >= 0.05cm, >=stealth] (2) to (4);
\draw [ ->, shorten <= 0.05cm, shorten >= 0.05cm, >=stealth] (4) to (6);
\draw [ ->, shorten <= 0.05cm, shorten >= 0.05cm, >=stealth] (2) to[out=-135,in=15] (6);
\draw [ ->, shorten <= 0.05cm, shorten >= 0.05cm, >=stealth] (6) to[out=45,in=-165] (2);
\draw [ ->, shorten <= 0.05cm, shorten >= 0.05cm, >=stealth] (6) to (1);

\filldraw (-1,4.75) circle (2pt);
\draw (-1.3,4.75) node{\small $1$};
\filldraw (1,4.75) circle (2pt);
\draw (1.3,4.75) node{\small $2$};
\filldraw (2,3) circle (2pt);
\draw (2.3,3) node{\small $3$};
\filldraw (1,1.25) circle (2pt);
\draw (1.3,1.25) node{\small $4$};
\filldraw (-1,1.25) circle (2pt);
\draw (-1.3,1.25) node{\small $5$};
\filldraw (-2,3) circle (2pt);
\draw (-2.3,3) node{\small $6$};

\end{tikzpicture}
\caption{The digraph of a Toeplitz matrix $T_6\langle 2,4;4,5\rangle$.}\label{fig:t62445}\end{center}
\end{figure}

We present an upper bound for the competition index of a  Toeplitz matrix $T_n\langle S;T \rangle$ with $s_{k_1} \le n-t_1$, $t_{k_2} \le n-s_1$.

\begin{theorem} \label{thm:compindex}
Let $A= T_n\langle S;T \rangle$ with $\max S + \min T \le n$, $\min S + \max S \le n$, $d = \gcd(S+T)$, and $d' = \gcd (d,s_1)$.
If the principal submatrix of $AA^T$ determined by the rows and columns indexed by $\{v \in [n] \mid v \equiv i \pmod d\}$ is irreducible for each $1 \le i \le d$,
then the competition index of $A$ is at most \[2\left( \left\lceil n/d \right\rceil -1 \right) \left( \max\left\{\left\lceil\frac{t_{k_2}}{s_1}\right\rceil,\left\lceil\frac{s_{k_1}}{t_1}\right\rceil\right\} +1 \right)+2(s_1+t_1).\]
\end{theorem}

\begin{proof}
Let $D$ be the digraph of $A$
and take two distinct vertices $u$ and $v$ in $D$.
If $u \not\equiv v \pmod d$, then
$u$ and $v$ are not adjacent in $C^m(D)$ for any positive integer $m$ by Proposition~\ref{prop:adj}.

We denote by $G_i$ the subgraph of $C(D)$ induced by the vertex set $\{v \in [n] \mid v \equiv i \pmod d\}$ for $1 \le i \le d$.
Then $G_i$ is a component of $C(D)$ for any $1 \le i \le d$ by the hypothesis.

Now suppose $u \equiv v \pmod d$.
Then $u$ and $v$ belong to $G_j$ for some $j \in [d]$ and so there is a path from $u$ to $v$ in $C(D)$.
Take one of such paths and let $v_0v_1\cdots v_\ell$ be its sequence with $v_0 = u$ and $v_\ell = v$.
Then there is a common out-neighbor $w_i$ of $v_i$ and $v_{i+1}$ for $0 \le i \le \ell-1$.
We note that $w_i - v_i= s_j$ or $-t_k$ for some $j \in [k_1]$, $k \in [k_2]$.
 For each $1 \le i \le k_1$, let $a_i$ (resp.\ $a'_i$) be the number of $s_i$-arcs among $(v_0,w_0), \ldots, (v_{\ell-1},w_{\ell-1})$ (resp.\ $(v_1,w_0), \ldots, (v_\ell,w_{\ell-1})$), that is,
\[a_i=|\{j \mid w_j-v_j=s_i, 0 \le j \le \ell-1 \}| \quad \text{and} \quad
a'_i=|\{j \mid w_j-v_{j+1}=s_i, 0 \le j \le \ell-1\}|.\]
Similarly, for each $1 \le i \le k_2$, let
\[b_i=|\{j \mid w_j-v_j=-t_i, 0 \le j \le \ell-1\}| \quad \text{and} \quad
b'_i=|\{j \mid w_j-v_{j+1}=-t_i, 0 \le j \le \ell-1\}|.\]
Then
\begin{equation}\label{eq:length}
a_1+\cdots +a_{k_1} + b_1 +\cdots+b_{k_2}=a'_1 +\cdots+a'_{k_1} + b'_1 + \cdots + b'_{k_2} =\ell
\end{equation} and
\[v_\ell-v_0 =\sum_{i=0}^{\ell-1}(w_i-v_i)-\sum_{i=0}^{\ell-1} (w_i-v_{i+1})=
\left(\sum_{i=1}^{k_1} a_is_i-\sum_{i=1}^{k_2} b_it_i\right)-\left(\sum_{i=1}^{k_1} a'_is_i-\sum_{i=1}^{k_2} b'_it_i\right). \]
Thus
\begin{equation}\label{eq:uv}
v_0 + \sum_{i=1}^{k_1} a_is_i - \sum_{i=1}^{k_2} b_it_i = v_\ell + \sum_{i=1}^{k_1} a'_is_i -\sum_{i=1}^{k_2} b'_it_i.
\end{equation}

On the other hand, by Theorem~\ref{lem:walk}(a), there is a directed walk $W_u$ starting from $u$ such that $W_u$ contains exactly $a_i$ $s_i$-arcs and $b_j$ $t_j$-arcs for each $2 \le i \le k_1$ and $2\le j\le k_2$ with length at most \[
(a_2+\cdots +a_{k_1}+b_2+\cdots+b_{k_2})\left(\max\left\{\left\lceil\frac{t_{k_2}}{s_1}\right\rceil,\left\lceil\frac{s_{k_1}}{t_1}\right\rceil\right\}+1\right).\]
Let $a$ and $b$ be the numbers of $s_1$-arcs and $t_1$-arcs, respectively, on $W_u$.
Then there is a nonnegative integer $c$ such that $a+ct_1 \ge a_1$ and $b + cs_1 \ge b_1$.
Take $c$ be the smallest nonnegative integer among such integers.
Then $c = 0$ or $a+ct_1 < a_1+t_1$ or $b+cs_1 < b_1 + s_1$.
Let $a_1^* = a+ct_1$ and $b_1^* = b+cs_1$.
Then
\begin{align*}&u+a_1^* s_1+\sum_{i=2}^{k_1} a_is_i - b_1^* t_1-\sum_{i=2}^{k_2} b_it_i\\ & =u+(a_1^*-a) s_1+as_1+\sum_{i=2}^{k_1} a_is_i - (b_1^*-b) t_1-bt_1-\sum_{i=2}^{k_2} b_it_i \\
&=u+as_1+\sum_{i=2}^{k_1} a_is_i - bt_1-\sum_{i=1}^{k_2} b_it_i
\end{align*}
Then, by Theorem~\ref{lem:walk}(b), there is a directed walk $W^*_u$ starting from $u$ such that $W^*_u$ contains exactly $a_1^*$ $s_1$-arcs, $b_1^*$ $t_1$-arcs, $a_i$ $s_i$-arcs and $b_j$ $t_j$-arcs for each $2 \le i \le k_1$ and $2\le j\le k_2$ with length
\begin{equation}\label{eq:lu}
 \ell_u := a_1^*+a_2+\cdots+a_{k_1}+b_1^*+b_2+\cdots+b_{k_2}
\end{equation}
If $c= 0$, then $a_1^* = a$ and $b_1^* = b$, and so
\begin{equation}\label{ineq:lu1}
\ell_u = |W_u| \le (a_2+\cdots +a_{k_1}+b_2+\cdots+b_{k_2})\left(\max\left\{\left\lceil\frac{t_{k_2}}{s_1}\right\rceil,\left\lceil\frac{s_{k_1}}{t_1}\right\rceil\right\}+1\right).
\end{equation}
Suppose $a+ct_1 < a_1+t_1$.
Then $a_1^* - a_1 < t_1$ and $c \le \lceil \frac{p_1}{t_1}\rceil$.
Therefore
\begin{align*}
b_1^* - b_1 &\le b_1^* \\
&= b+cs_1\\
&\le a+b + s_1 \left\lceil \frac{a_1}{t_1} \right\rceil \\
&< (a_2+\cdots +a_{k_1}+b_2+\cdots+b_{k_2})\left(\max\left\{\left\lceil\frac{t_{k_2}}{s_1}\right\rceil,\left\lceil\frac{s_{k_1}}{t_1}\right\rceil\right\}\right)
+a_1 \frac{s_1}{t_1}+s_1 \\
& \le (a_1+\cdots+a_{k_1}+b_2+\cdots+b_{k_2})\left(\max\left\{\left\lceil\frac{t_{k_2}}{s_1}\right\rceil,\left\lceil\frac{s_{k_1}}{t_1}\right\rceil\right\}\right) +s_1.
\end{align*}
Thus
\begin{equation}\label{ineq:lu2}
\ell_u = \ell + (a_1^* -a_1) + (b_1^*-b_1) < \ell\left(\max\left\{\left\lceil\frac{t_{k_2}}{s_1}\right\rceil,\left\lceil\frac{s_{k_1}}{t_1}\right\rceil\right\}\right)+s_1+t_1.
\end{equation}
If $b+cs_1 < b_1 + s_1$, then $\ell_u < \ell\left(\max\left\{\left\lceil\frac{t_{k_2}}{s_1}\right\rceil,\left\lceil\frac{s_{k_1}}{t_1}\right\rceil\right\}\right)+s_1+t_1$ by a similar argument.
Hence, we have $\ell_u < \ell\left(\max\left\{\left\lceil\frac{t_{k_2}}{s_1}\right\rceil,\left\lceil\frac{s_{k_1}}{t_1}\right\rceil\right\}+1\right)+s_1+t_1$ for each cases.

Then $W^*_u$ has exactly $a_i$ $s_i$-arcs and $b_j$ $t_j$-arcs for $2 \le i \le k_1$ and $2 \le j \le k_2$.
Moreover, $W^*_u$ has $a_1^*$ $s_1$-arcs and $b_1^*$ $t_1$-arcs.

By Theorem~\ref{lem:walk}(b) again, one can make a directed walk $W^*_v$ starting from $v$ such that $W^*_v$ contains $a_1^{**} \ge a'_1$ $s_1$-arcs and $b_1^{**} \ge b'_1$ $t_1$-arcs and has length
\begin{align}
\ell_v&:= a_1^{**}+a'_2+\cdots+a'_{k_1}+b_1^{**}+b'_2+\cdots+b'_{k_2}\nonumber \\
&= \ell + (a_1^{**}-a'_1)+(b_1^{**}-b'_1)\label{eq:lv} \mbox{ (by \eqref{eq:length})}\\
&\le \ell\left(\max\left\{\left\lceil\frac{t_{k_2}}{s_1}\right\rceil,\left\lceil\frac{s_{k_1}}{t_1}\right\rceil\right\}+1\right)+s_1+t_1.\label{ineq:lv}
\end{align}

Let $z_u$ and $z_v$ be the termini of $W^*_u$ and $W^*_v$, respectively.
Then
\[
z_u = u+ a_1^* s_1 + a_2s_2+ \cdots + a_{k_1}s_{k_1} -b_1^* t_1 -b_2t_2 - \cdots - b_{k_2}t_{k_2}\]
and
\[
z_v = v+a_1^{**}s_1+a'_2s_2 +\cdots + a'_{k_1}s_{k_1} - b_1^{**} t_1 - b_2t_2 -\cdots - b_{k_2}t_{k_2}.
\]
Thus, by \eqref{eq:uv},
\[
z_u-z_v = ((a_1^* -a_1)-(a_1^{**} -a'_1))s_1 - ((b_1^* -b_1)-(b_1^{**} -b'_1))t_1.
\]
Therefore
\[
z_u + (a_1^{**}-a'_1)s_1 - (b_1^{**}-b'_1)t_1 = z_v + (a_1^* - a_1)s_1 - (b_1^* - b_1)t_1=:z.
\]
Then, by Lemma~\ref{lem:divis}, there exist a directed walk $W'_u$ from $z_u$ to $z$ of length $a_1^{**} -a'_1 + b_1^{**}-b'_1$ and a directed walk $W'_v$ from $z_v$ to $z$ of length $a_1^*-a_1 + b_1^*-b_1$ in $T_n \langle s_1 ; t_1 \rangle$.
Since $T_n \langle s_1; t_1 \rangle$ is a subdigraph of $D$, $W'_u$ and $W'_v$ are also directed walks in $D$.
Thus $W^*_u\rightarrow W'_u$ is a $(u,z)$-directed walk of length $\ell_u+(a_1^{**}-a'_1)+(b_1^{**}-b'_1)$ and $W^*_v\rightarrow W'_v$ is a $(v,z)$-directed walk of length $\ell_v+(a_1^*-a_1)+(b_1^*-b_1)$ in $D$.
By \eqref{eq:lu} and \eqref{eq:lv},
\begin{align*}
\ell_u+(a_1^{**}-a'_1)+(b_1^{**}-b'_1) &= m+ (a_1^*-a_1)+(b_1^*-b_1)+(a_1^{**}-a'_1)+(b_1^{**}-b'_1) \\
&= \ell_v + (a_1^*-a_1)+(b_1^*-b_1).
\end{align*}
Moreover, by \eqref{ineq:lu2} and \eqref{ineq:lv},
\begin{align*}
\ell_u+(a_1^{**}-a'_1)+(b_1^{**}-b'_1) &\le
\ell_u+\ell_v \\
&\le 2 \ell \left(\max\left\{\left\lceil\frac{t_{k_2}}{s_1}\right\rceil,\left\lceil\frac{s_{k_1}}{t_1}\right\rceil\right\}+1\right)+2(s_1+t_1).
\end{align*}
Since there is at most $\left\lceil n/d \right\rceil$ vertices in $G_i$ for each $1 \le i \le d$, we have $\ell \le \left\lceil n/d\right\rceil -1$.
Thus $W^*_u\rightarrow W'_u$ and $W^*_v \rightarrow W'_v$ has the same length at most
\[
2(\left\lceil n/d \right\rceil-1) \left(\max\left\{\left\lceil\frac{t_{k_2}}{s_1}\right\rceil,\left\lceil\frac{s_{k_1}}{t_1}\right\rceil\right\}+1\right)+2(s_1+t_1).
\]
Hence we have shown that the competition index of $A$ is at most
\[
2\left( \left\lceil n/d \right\rceil -1 \right) \left( \max\left\{\left\lceil\frac{t_{k_2}}{s_1}\right\rceil,\left\lceil\frac{s_{k_1}}{t_1}\right\rceil\right\} +1 \right)+2(s_1+t_1). \qedhere
\]
\end{proof}

\section{A Proof of Theorem~\ref{lem:pqr}}\label{se:pf}

\subsection{$(\exists m^* \in \mathbb{N}) [P_{m^*} \subseteq Q_{m^*}]$}

We will show that there is a positive integer $m^\star$ such that $P_{m^\star} \subseteq Q_{m^\star}$.
Prior to that, we take number theoretic approach to deduce the following two results.

B\'ezout's identity  asserts that if ${\rm gcd}(a,b)=d$ then there exist integers $x$ and $y$ such that $ax + by = d$.
Thus, if $d={\rm gcd}(S+T)$ then by Proposition~\ref{prop:gcd} together with B\'ezout's identity there exit some integers $\alpha_{i,j}, \beta_{i,j}, \gamma_{i,j}$ such that
\begin{eqnarray}\label{d}
d = \sum_{1 \le i<j \le k_1} \alpha_{i,j}(s_j-s_i) + \sum_{1 \le i < j \le k_2} \beta_{i,j}(t_j-t_i) + \sum_{1 \le i \le k_1, 1\le j \le k_2} \gamma_{i,j}(s_i+t_j).
\end{eqnarray}

\begin{lemma}\label{prop:gcd-1} Let $d={\rm gcd}(S+T)$. Then there exist integers $a_i,b_i$ such that
\begin{equation*}
d= \sum_{i = 1}^{k_1} a_i s_i
  - \sum_{i = 1}^{k_2}b_i t_i
\quad\text{ and }\quad\sum_{i = 1}^{k_1}a_i
                   +\sum_{i = 1}^{k_2}b_i = 0.
\end{equation*}
\end{lemma}
\begin{proof}
By noting that  $s_j- s_i = (s_j-s_{j-1})+(s_{j-1}-s_{j-2})+\cdots + (s_{i+1}-s_i)$ and $t_j-t_i = (t_j-t_{j-1})+\cdots + (t_{i+1}-t_i)$, the right-hand side of (\ref{d}) gives\begin{align*}
\displaystyle
d &= \sum_{1 \le i< j \le k_1}\alpha_{i,j} \sum_{k=i+1}^j (s_k-s_{k-1}) + \sum_{1 \le i < j \le k_2} \beta_{i,j} \sum_{k=i+1}^j (t_k-t_{k-1}) \\&\qquad+ \sum_{1 \le i \le k_1, 1\le j \le k_2}  \gamma_{i,j} \left(\sum_{k=2}^i(s_k-s_{k-1}) + \sum_{k=2}^j (t_k-t_{k-1})+ (s_1+t_1).\right) \\
&= \sum_{i=2}^{k_1} \alpha_i (s_i-s_{i-1}) + \sum_{i=2}^{k_2} \beta_i(t_i-t_{i-1}) + \gamma(s_1+t_1)
\end{align*}
where
\[
\alpha_i = \sum_{1 \le j <  i \le k\le k_1} \alpha_{j,k} +  \sum_{2 \le i \le j \le k_1, 1 \le k \le k_2} \gamma_{j,k},\]
\[\beta_i = \sum_{1 \le j < i \le k \le k_2} \beta_{j,k} +\sum_{1 \le j \le k_1, 2 \le i \le k \le k_2} \gamma_{j,k}, \] and \[
\gamma = \sum_{1 \le i \le k_1, 1 \le j \le k_2} \gamma_{i,j}.
\]
Therefore
\begin{align*}
\displaystyle
d= &\sum_{i = 2}^{k_1} \alpha_i (s_i-s_{i-1})
+ \sum_{i = 2}^{k_2} \beta_i (t_i-t_{i-1}) + \gamma(s_1 +t_1)\\
= &(\gamma - \alpha_2)s_1 + \sum_{i=2}^{k_1-1}(\alpha_i - \alpha_{i+1})s_i
 +\alpha_{k_1}s_{k_1}\\ & - (\beta_2 - \gamma)t_1 -\sum_{i=2}^{k_2-1}(\beta_{i+1}-\beta_i)t_i
 -(-\beta_{k_2})t_{k_2}.\end{align*}
Thus there are some integers $a_i,b_i$ such that
\begin{equation*}
d= \sum_{i = 1}^{k_1} a_i s_i
  - \sum_{i = 1}^{k_2}b_i t_i
\quad\text{ and }\quad\sum_{i = 1}^{k_1}a_i
                   +\sum_{i = 1}^{k_2}b_i = 0. \qedhere
\end{equation*}
\end{proof}

\begin{theorem}\label{lem:cons}
Let $d={\rm gcd}(S+T)$.
For a given positive integer $k$, there is an integer $r$ such that for any $j \in [k]$, \[\displaystyle r + jd =\sum_{i = 1}^{k_1} a_{i,j}s_i
- \sum_{i = 1}^{k_2}b_{i,j}t_i\] for some nonnegative integers
$a_{i,j}, b_{i,j}$ with the constant $\displaystyle f(j) := \sum_{i = 1}^{k_1} a_{i,j} + \sum_{i = 1}^{k_2}b_{i,j}$.
\end{theorem}

\begin{proof}
 By Lemma~\ref{prop:gcd-1},
there are some integers $a_i,b_i$ such that
\[\displaystyle
d= \sum_{i = 1}^{k_1} a_i s_i
  - \sum_{i = 1}^{k_2}b_i t_i
\quad\text{ and }\quad\sum_{i = 1}^{k_1}a_i
                   +\sum_{i = 1}^{k_2}b_i = 0.
 \]
Thus
\[\displaystyle
jd = \sum_{i=1}^{k_1} (ja_{i})s_i
   - \sum_{i = 1}^{k_2} (jb_{i})t_i
\]
for each $j \in [k]$.
By adding
$\displaystyle r:=\sum_{i = 1}^{k_1} (k|a_i|)s_i -\sum_{i=1}^{k_2}(k|b_i|)t_i$
to both sides of the above equality, we have
\[\displaystyle
r+jd = \sum_{i=1}^{k_1} a_{i,j}s_i -
       \sum_{i = 1}^{k_2} b_{i,j}t_i
\]
where $a_{i,j} = ja_i+ k|a_i|$ and
$b_{i,j} = jb_i+k|b_i|$,
which are nonnegative for each $j \in[k]$.
Moreover,
\begin{align*}
\displaystyle \sum_{i=1}^{k_1} a_{i,j}+\sum_{i=1}^{k_2}b_{i,j}
&= j\left( \sum_{i=1}^{k_1}a_i + \sum_{i=1}^{k_2}b_i\right) +
b\left( \sum_{i=1}^{k_1}|a_i| + \sum_{i=1}^{k_2}|b_i|\right) \\
&= k\left( \sum_{i=1}^{k_1}|a_i| + \sum_{i=1}^{k_2}|b_i|\right)\end{align*}
for each $j \in [k]$ and so $f(j)$ is constant.
\end{proof}

\begin{theorem}
There exists a positive integer $m^\star$ such that $P_{m^\star} \subseteq Q_{m^\star}$.
\end{theorem}
\begin{proof}
For given sets of positive integers $S$ and $T$,
let $d={\rm gcd}(S+T)$.
By Lemma~\ref{prop:gcd-1},
there are some integers $a_i,b_i$ such that
\[\displaystyle
d= \sum_{i = 1}^{k_1} a_i s_i
  - \sum_{i = 1}^{k_2}b_i t_i
\quad\text{ and }\quad\sum_{i = 1}^{k_1}a_i
                   +\sum_{i = 1}^{k_2}b_i = 0.
 \]

Take a sufficiently large positive integer $k$ such that
\begin{equation}\label{eq:qd}
(k-1)d >2(n-1)+t_1+s_1.	
\end{equation}
By Theorem~\ref{lem:cons},
there are an integer $r$ and nonnegative integers
$a_{i,j}, b_{i,j}$ such that for each $j\in[k]$,
\begin{equation} \label{eq:a}
r+ jd = \sum_{i=1}^{k_1} a_{i,j}s_i - \sum_{i=1}^{k_2} b_{i,j}t_i
\end{equation}
with $\displaystyle \sum_{i = 1}^{k_1} a_{i,j} + \sum_{i=1}^{k_2}b_{i,j}$ constant.

{\it Case 1.} $r+d<-(n-1)$.
Then there is the smallest positive integer $c$ such that
\begin{equation}\label{eq:cs}
r+d+cs_1 \ge -(n-1).
\end{equation}
Let \[r' = r+(c-1)s_1, \quad a'_{1,j} = a_{1,j}+c-1, \quad \text{and}\quad a'_{i,j} = a_{i,j}\] for each $2 \le i \le k_1$ and $1 \le j\le k$.
Then, by \eqref{eq:a},
\begin{equation}	\label{eq:equality}
m:  =\sum_{i=1}^{k_1}a'_{i,j} + \sum_{i=1}^{k_2}b_{i,j} = \sum_{i=1}^{k_1}a_{i,j}+\sum_{i=1}^{k_2}b_{i,j}+c-1,
\end{equation}
which is constant, and
\begin{equation} \label{eq:equality1}
r'+ jd  = \sum_{i=1}^{k_1} a'_{i,j}s_i - \sum_{i=1}^{k_2} b_{i,j}t_i
\end{equation} for each $j \in [k]$.
By Proposition~\ref{lem:xiperiod},
\[r' \equiv  ms_1 \pmod{d}.\]
Moreover, \[r'+d<-(n-1)\] by the choice of $c$.
On the other hand, by \eqref{eq:qd} and \eqref{eq:cs},
\[ r'+bd = (r'+d+s_1) + ((k-1)d-s_1) \ge (r+d+cs_1) +2(n-1) \ge  n-1.\]

Take $p\in P_m$.
Then $p \equiv ms_1 \equiv r' \pmod d$ and
 \[r'+d \le -(n-1) \le p \le n-1 \le r'+kd.\]
Therefore $p \in \{r'+d,\ldots,r'+kd\}$.
Thus $p\in Q_m$ by \eqref{eq:equality} and \eqref{eq:equality1}.

{\it Case 2.} $r+d\ge -(n-1)$.
Then there is the smallest positive integer $c'$ such that
\begin{equation} \label{eq:cprime}
r+d-c't_1 < -(n-1).	
\end{equation}
Let
\[r'' =  r-(c'-1)t_1, \quad b'_{1,j} = b_{1,j}+ c'-1, \quad \text{and} \quad b'_{i,j} = b_{i,j}\]
 for each $2 \le i \le k_2$ and $1 \le j \le k$.
 Then, by \eqref{eq:a},
\begin{equation} \label{eq:mprime}
m:=r''+jd = \sum_{i=1}^{k_1} a_{i,j}s_i - \sum_{i=1}^{k_2} b'_{i,j}t_i
\end{equation}  with $\displaystyle \sum_{i = 1}^{k_1} a_{i,j} + \sum_{i=1}^{k_2}b'_{i,j}$ constant
for each $j \in [k]$.
Then $r'' \equiv ms_1 \pmod d$ by Proposition~\ref{lem:xiperiod}.
Moreover, $r''+d < -(n-1)$ by the choice of $c'$ and, by \eqref{eq:qd} and \eqref{eq:cprime},
 \[
r''+kd = (r''+d+t_1)+((k-1)d-t_1) \ge (r+d-c't_1)+ 2(n-1) \ge n-1.
\]

Take $p\in P_m$.
Then $p \equiv ms_1 \equiv r'' \pmod d$ and $r''+d \le p \le r''+bd$, i.e.
$p \in \{r''+d,\ldots,r''+bd\}$.
Thus $p\in Q_m$ by \eqref{eq:mprime}.
Hence, in both cases, we have shown that there is a positive integer $m^*$ such that $P_{m^*} \subseteq Q_{m^*}$.\end{proof}

\subsection{Taking a candidate for $M$ satisfying Theorem~\ref{lem:pqr}}

 We have shown there exists a positive integer $m^*$ such that $P_{m^*} \subseteq Q_{m^*}$.
Now fix $p\in P_{m^*}$.
Then $p \in Q_{m^*}$ and there are nonnegative integers
$a_{i,p}, b_{i,p}$ such that
\begin{equation}\label{eq:comb}
\sum_{i=1}^{k_1}a_{i,p}+\sum_{i=1}^{k_2}b_{i,p}=m^* \quad \text{and} \quad p = \sum_{i=1}^{k_1}a_{i,p}s_i - \sum_{i=1}^{k_2}b_{i,p}t_i.
\end{equation}
By Theorem~\ref{lem:walk}(a) and \eqref{eq:comb},
\begin{itemize}
\item[($\star$)]for each vertex $v \in V(D)$,
there exists a directed walk
$W_{v,p}$ starting from $v$ such that $W_{v,p}$ contains exactly $a_{i,p}$ $s_i$-arcs and
$b_{j,p}$ $t_j$-arcs for each $2 \le i \le k_1$ and $2 \le j \le k_2$ where $\sum_{i=1}^{k_1}a_{i,p}+\sum_{i=1}^{k_2}b_{i,p}=m^*$ and $p = \sum_{i=1}^{k_1}a_{i,p}s_i - \sum_{i=1}^{k_2}b_{i,p}t_i$.
\end{itemize}
Let $a_{v,p}$ (resp. $b_{v,p}$) be the number of $s_1$-arcs (resp. $t_1$-arcs) in $W_{v,p}$.
Then there exists some positive integer $c$ such that
\[
a_{1,p}+c t_1 \ge \max_{v\in V(D)}a_{v,p} \quad \text{and}\quad
b_{1,p}+c s_1 \ge \max_{v\in V(D)}b_{v,p}.
\]
Now, by \eqref{eq:comb},
\begin{equation} \label{eq:another}
p = \sum_{i=2}^{k_1}a_{i,p}s_i - \sum_{i=2}^{k_2}b_{i,p}t_i + (a_{1,p}+ct_1)s_1 - (b_{1,p}+cs_1)t_1.
\end{equation}
We note that
\[\sum_{i=2}^{k_1}a_{i,p}+\sum_{i=2}^{k_2}b_{i,p}+(a_{1,p}+ct_1)+(b_{1,p}+cs_1)=\sum_{i=2}^{k_1}a_{i,p}+\sum_{i=1}^{k_2}b_{i,p}+c(s_1+t_1)=m^*+c(s_1+t_1).\]
Let
\[
M=m^*+c(s_1+t_1)\] which is a candidate for Theorem~\ref{lem:pqr}.
Then, since $d \mid (s_1+t_1)$,
\begin{equation}\label{eq:pmpm}
P_{m^*} = P_M.
\end{equation}
\subsection{$(\forall {i \ge M}) [P_i \subseteq R_i]$ }
Since $R_i \subseteq Q_i \subseteq P_i$ for any positive integer $i$ by \eqref{eq:pqr}, $(\forall {i \ge M})[P_i \subseteq R_i]$ implies $(\forall {i \ge M})[P_i=Q_i=R_i]$.
 Thus proving that $P_i \subseteq R_i$ for any integer $i \ge M$ will  complete the proof of Theorem~\ref{lem:pqr}.

Now fix $\ell \ge M$.
By Lemma~\ref{lem:abcd}(c),
\begin{align*}
P_{\ell} & =\{p \in \mathcal{I}_n\mid p=p_0+a s_1-b t_1 \text{ for some $p_0 \in P_M$ and  nonnegative integers $a$ and $b$  with }  \\
& \hskip1.5cm a+b=\ell-M\}.	
\end{align*}
To show $P_\ell \subseteq R_\ell$, take $p \in P_\ell$.
Then there exist some $p_0 \in P_M$, and nonnegative integers $a$ and $b$ such that \[
p = p_0 + a s_1 - b t_1\quad \text{and}\quad a+b = \ell - M.
\]
Furthermore, by \eqref{eq:comb} and \eqref{eq:pmpm},
 \[p_0=\sum_{i=1}^{k_1}a_{i,0}s_i - \sum_{j=1}^{k_2}b_{j,0}t_j\]
where $a_{i,0}$, $b_{j,0}$ are nonnegative integers with
\[\sum_{i=1}^{k_1}a_{i,0}+\sum_{j=1}^{k_2}b_{j,0}=m^\star.\]
Therefore
\[p=\sum_{i=2}^{k_1}a_{i,0}s_i - \sum_{j=2}^{k_2}b_{j,0}t_j +(a_{1,0}+ct_1+a)s_1 - (b_{1,0}+cs_1+b)t_1.
\]
Now,
\begin{align*}
\sum_{i=2}^{k_1}a_{i,0} & + \sum_{j=2}^{k_2}b_{j,0} +(a_{1,0}+ct_1+a)+ (b_{1,0}+cs_1+b) \\ & =\sum_{i=1}^{k_1}a_{i,0}+\sum_{j=1}^{k_2}b_{j,0}+c(s_1+t_1)+(a+b) \\ & =m^\star+c(s_1+t_1)+(\ell-M)=M+(\ell-M)=\ell.
\end{align*}
To show $p \in R_\ell$, take vertices $v$ and $v+p$ of $D$.
Then there exists a directed walk $W_{v,0}$ that contains exactly $a_{i,0}$ $s_i$-arcs and $b_{j,0}$ $t_j$-arcs for each $2 \le i \le k_1$ and $2 \le j \le k_2$  by ($\star$).
Now we need to extend $W_{v,0}$ by attaching to its terminus
\[w:=v+\sum_{i=2}^{k_1}a_{i,0}-\sum_{j=2}^{k_2}b_{j,0}\]
a directed walk consisting of $s_1$-arcs exactly as many as $a_{1,0} + ct_1 + a$ and $t_1$-arcs exactly as many as $b_{1,0} + cs_1+b$.
We note that
\[w+(a_{1,0}+ct_1+a)s_1-(b_{1,0}+cs_1+b)t_1=v+p.\]
Since $v+p$ is a vertex in $D$, $1 \le v+p \le n$ and so the existence of such a directed walk to be attached to $W_{v,0}$ is guaranteed by applying Theorem~\ref{lem:walk}(b).
Since $v$ was arbitrarily chosen, $p \in R_\ell$.
In addition, $p$ was arbitrarily chosen, and so we finally have
\begin{equation*}
P_\ell \subseteq R_\ell. \tag*{$\square$}
\end{equation*}

\section{Acknowledgement}
This work was partially supported by Science Research Center Program through the National Research Foundation of Korea(NRF) Grant funded by the Korean Government (MSIP)(NRF-2016R1A5A1008055). G.-S. Cheon was partially supported by the NRF-2019R1A2C1007518. Bumtle Kang was partially supported by the NRF-2021R1C1C2014187. S.-R. Kim and H. Ryu were partially supported by the Korea government (MSIP) (NRF-2017R1E1A1A03070489 and NRF-2022R1A2C1009648).

\end{document}